\documentclass[11pt]{amsart}
\usepackage{amsmath}
\usepackage{amssymb}
\usepackage{amscd}

\def\NZQ{\mathbb}               
\def\NN{{\NZQ N}}

\def\ZZ{{\NZQ Z}}

\newtheorem{Theorem}{Theorem}[section]

\newtheorem{Corollary}[Theorem]{Corollary}

\let\epsilon\varepsilon
\let\phi=\varphi
\let\kappa=\varkappa

\begin{document}

\title{A stronger local monomialization theorem}
\author{Steven Dale Cutkosky}
\thanks{partially supported by NSF}

\address{Steven Dale Cutkosky, Department of Mathematics,
University of Missouri, Columbia, MO 65211, USA}
\email{cutkoskys@missouri.edu}

\dedicatory{Dedicated to Professor Winfried Bruns on the occasion of his 70th birthday}

\begin{abstract} In this article we prove stronger versions of  local monomialization.
\end{abstract}

\maketitle

In this note we  derive  extensions of the monomialization theorems for morphisms of varieties in \cite{Ast} and \cite{LMTE}. I thank Jan Denef for conversations on this topic, suggesting that I make improvements of this type, and explaining applications of these theorems. A global ``weak'' monomialization theorem is established in \cite{ADK} by Abramovich, Denef and Karu, generalizing an earlier theorem by Abramovich and Karu in \cite{AK}.  A monomialization is ``weak'' if the modifications used have no further requirements; in a monomialization all modifications must be products of blow ups of nonsingular sub varieties. In this note we show that  a local monomialization can be found which satisfies the extra local statements obtained in \cite{ADK}. In \cite{De} and \cite{De1}
some comments are made about how the results of this paper can be used.

The techniques in this paper come from the theory of resolution of singularities. Some basic references in this subject are \cite{H0}, \cite{BM1}, \cite{BV}, \cite{CP1}, \cite{CP2}, \cite{CP3}, \cite{EV}, \cite{Ha1}, \cite{T}. In this paper we assume that the ground field has characteristic zero. Counterexamples to local monomialization in positive characteristic are given in \cite{C2}. A proof of local monomialization, within the context of analytic geometry, is given in \cite{LM} for germs of real and complex analytic maps.

\section{A stronger local monomomialization theorem for algebraic morphisms}

In this section we state and prove an extension Theorem \ref{TheoremF} of the local monomialization theorem Theorem 1.4 \cite{LMTE}. 

Suppose that $K$ is an algebraic function field over a field $k$. A local ring $R$ is an algebraic local ring of $K$ if $R$ is a subring of $K$, the quotient field of $R$ is $K$ and $R$ is essentially of finite type over $k$.

\begin{Theorem}\label{TheoremF} Suppose that $k$ is a field of characteristic zero, $K\rightarrow K^*$ is a (possibly transcendental) extension of algebraic function fields over $k$, and that $\nu^*$ is a valuation of $K^*$ which is trivial on $k$. Further suppose that $R$ is an algebraic local ring of $K$ and $S$ is an algebraic local ring of $K^*$ such that $S$ dominates $R$ and $\nu^*$ dominates $S$. Suppose that $I$ is a nonzero ideal of $S$. Then there exist sequences of monoidal transforms $R\rightarrow R'$ and $S\rightarrow S'$ along $\nu^*$ such that $R'$ and $S'$ are regular local rings, $S'$ dominates $R'$, there exist regular parameters $(y_1,\ldots,y_n)$ in $S'$, $(x_1,\ldots,x_m)$ in $R'$, units $\delta_1,\ldots,\delta_m\in S'$ and an $m\times n$ matrix $(c_{ij})$ of nonnegative integers such that $(c_{ij})$ has rank $m$, and
$$
x_i=\prod_{j=1}^ny_j^{c_{ij}}\delta_i
$$
for $1\le i\le m$.
Further, we have that 
\begin{enumerate}
\item[1)] $S'$ is a local ring of the blowup of an ideal $J$ of $S$ such that $JS'=(y_1^{\alpha_1}\cdots y_n^{\alpha_n})$ for some $\alpha_1,\ldots,\alpha_n\in \NN$.
\item[2)]   $I S'=(y_1^{\beta_1}\cdots y_n^{\beta_n})$ for some $\beta_1,\ldots,\beta_n\in \NN$.
\end{enumerate}
\end{Theorem}

The proof of Theorem 1.4 \cite{LMTE} is given in \cite{Ast} and \cite{LMTE}. Also see \cite{C3} for some errata. The new part of Theorem 1.1 is the addition of the conclusions 1) and 2).
We now explain the changes in this proof which must be made to obtain the stronger result Theorem \ref{TheoremF}. Theorem \ref{TheoremF} is a consequence of  Theorem \ref{TheoremM4} which is proven at the end of this section.

We first indicate changes required in the proofs of \cite{Ast}  to obtain Theorem \ref{TheoremF} in the case of a finite extension of function fields.
In the construction of Theorem 5.1 \cite{Ast}, we have that $\nu(y_1'),\ldots,\nu(y_s')$ are rationally independent. We first settle the case when the quotient field of $S$ is finite over the quotient field of $R$ and the valuation $\nu$ has rank 1.

\vskip .2truein
\begin{Theorem} \label{TheoremM1}
Suppose that $R\rightarrow S$ and $R'\rightarrow S'$ satisfy the assumptions and conclusions of Theorem 5.1 \cite{Ast} and $h\in S'$ is nonzero. Then there exist sequences of monoidal transforms $R'\rightarrow R''$ and $S'\rightarrow S''$ along $\nu$ such that $S''$ dominates $R''$, $R''$ has regular parameters $x_1'',\ldots,x_n''$, $S''$ has regular parameters $y_1'',\ldots,y_n''$ having the monomial form of the conclusions of Theorem 5.1 \cite{Ast} and
$$
h=(y_1'')^{e_1}\cdots(y_s'')^{e_s}\overline u
$$
where $e_1,\ldots,e_s\in \NN$ and  $\overline u\in S''$ is a unit.
\end{Theorem}

\begin{proof}  This is an immediate consequence of Theorems 4.8 and 4.10 of \cite{Ast}
 \end{proof}

\begin{Corollary}\label{CorM2} Suppose that $R\rightarrow S$ and $R'\rightarrow S'$ satisfy the assumptions and conclusions of Theorem 5.1 \cite{Ast} and $I\subset S$ is a nonzero ideal.  Then there exist  sequences of monoidal transforms $R'\rightarrow R''$ and $S'\rightarrow S''$ along $\nu$ such that $S''$ dominates $R''$, $R''$ has regular parameters $x_1'',\ldots,x_n''$, $S''$ has regular parameters $y_1'',\ldots,y_n''$ having the monomial form of the conclusions of Theorem 5.1 \cite{Ast} and the following holds.
\begin{enumerate}
\item[1)] $S''$ is a local ring of the blowup of an ideal $J$ of $S$ such that $JS''=((y_1'')^{a_1}\cdots (y_s'')^{a_s})$ for some $a_1,\ldots,a_s\in \NN$.
\item[2)]   $I S''=((y_1'')^{b_1}\cdots (y_s'')^{b_s})$ for some $b_1,\ldots,b_s\in \NN$.
\end{enumerate}
\end{Corollary}

\begin{proof}  Let $K$ be an ideal in $S$ such that $S'$ is a local ring of the blow up of $K$. Let $KS'=(f_0)$ and $IS'=(f_1,\ldots,f_l)$. Let $h=\prod_{j=0}^lf_i$. By Theorem \ref{TheoremM1}, there exist sequences of monoidal transforms $R'\rightarrow R''$ and $S'\rightarrow S''$ along $\nu$ such that $R''\rightarrow S''$ and $h$ satisfy the conclusions of Theorem \ref{TheoremM1}.
Now by Lemma 4.2  and Remark 4.1  \cite{Ast}, there exists a sequence of monoidal transforms  $S''\rightarrow S(1)$ such that $R''\rightarrow S(1)$ has a monomial form as in the conclusion of Theorem 5.1 \cite{Ast}, $KS(1)=(y_1(1)^{e_{1}}\cdots y_s(1)^{e_s})$ and $IS(1)=(y_1(1)^{c_{1}}\cdots y_s(1)^{c_s})$ for some $e_i$ and $c_j$ in $\NN$.

Now $S'\rightarrow S(1)$   is a product of Perron transforms (as defined in Section 4.1 \cite{Ast}). Thus $S(1)$  is a local ring of the blow up of an  ideal $L$ in $S'$ such that 
$$
LS(1)=(y_1(1)^{g_1(i+1)}\cdots y_s(1)^{g_s(i+1)})
$$
 for some $g_j(i+1)\in \NN$. There exists a positive integer $\beta$ such that $S(1)$ is a local ring of the blow up of an ideal $J$ of $S$ such that $JS(1)=K^{\beta}LS(1)$
 (this can be verified using the universal property of blowing up) and  the corollary follows.
\end{proof}

We now prove the case when the quotient field of $S$ is finite over the quotient field of $R$ and the  valuation ring has arbitrary rank. We use the notation of Theorem 5.3 \cite{Ast}.

\begin{Theorem}\label{TheoremM3} 
Suppose that $R\rightarrow S$ satisfies the assumptions  of Theorem 5.3 \cite{Ast} and $I$ is a nonzero ideal in $S$. Then there exist sequences of monoidal transforms $R\rightarrow R'$ and $S\rightarrow S'$ such that $R'\rightarrow S'$ satisfies the conclusions of Theorem 5.3 \cite{Ast} and
\begin{enumerate}
\item[1)] $S'$ is a local ring of the blowup of an ideal $J$ of $S$ such that 
$$
JS'=(\prod_{j=0}^{r-1}\prod_{l=1}^{s_{j+1}}(w_{t_1+\cdots+t_j+l})^{\epsilon_{jl}})
$$
 for some $\epsilon_{jl}\in \NN$ (with the convention that $t_1+\cdots+t_0=0$).
\item[2)]   $I S'=(\prod_{j=0}^{r-1}\prod_{l=1}^{s_{j+1}}(w_{t_1+\cdots+t_j+l})^{\gamma_{jl}})$ for some $\gamma_{jl}\in \NN$.
\end{enumerate}
\end{Theorem}
\vskip .2truein
To obtain Theorem \ref{TheoremM3}, we must modify the proof of Theorem 5.3 \cite{Ast} as follows. 
\vskip .2truein
On line 1 of page 117, replace the reference to Theorem 5.1 \cite{Ast} with Corollary \ref{CorM2}.
\vskip .2truein
Insert the following at the end of line 15 on page 119 (after ``for $1\le i\le \lambda$''): ``By our construction, $S''$ is a local ring of the blow up of an ideal $J$ of $S$ such that 
$$
JS''_{q''_{r-1}}=(\prod_{j=0}^{r-2}\prod_{l=1}^{s_{j+1}}(y''_{t_1+\cdots+t_j+l})^{\alpha_{jl}})
$$
for some $\alpha_{jl}\in \NN$ so 
$$
JS''=H''(\prod_{j=0}^{r-2}\prod_{l=1}^{s_{j+1}}(y''_{t_1+\cdots+t_j+l})^{\alpha_{jl}})
$$
where $H''$ is an ideal in $S''$ such that $H''S''_{q''_{r-1}}=S''_{q''_{r-1}}$. By our construction,
$$
IS''_{q''_{r-1}}=(\prod_{j=0}^{r-2}\prod_{l=1}^{s_{j+1}}(y''_{t_1+\cdots+t_j+l})^{\beta_{jl}})
$$
for some $\beta_{jl}\in \NN$ so 
$$
IS''=K''(\prod_{j=0}^{r-2}\prod_{l=1}^{s_{j+1}}(y''_{t_1+\cdots+t_j+l})^{\alpha_{jl}})
$$
where $K''$ is an ideal in $S''$ such that $K''S''_{q''_{r-1}}=S''_{q''_{r-1}}$.''
\vskip .2truein
After line 11 of page 120 (after ``for $1\le i\le \lambda$''), insert: ``We also may  obtain, using Theorem \ref{TheoremM1} and the argument of the proof of Corollary \ref{CorM2} (above) that
$$
H''U'=(\overline y_{\lambda+1}^{\gamma_1}\cdots\overline y_{\lambda+s_r}^{\gamma_{s_r}})
$$
and
$$
K''U'=(\overline y_{\lambda+1}^{\delta_1}\cdots\overline y_{\lambda+s_r}^{\delta_{s_r}}).''
$$
\vskip .2truein
Insert the following at the end of line -2 of page 121 (at the end of the proof): ``We have
$$
H''\overline S(m'+1)=(\overline y_{\lambda+1}(m'+1)^{\gamma_1}\cdots \overline y_{\lambda+s_r}(m'+1)^{\gamma_{s_r}}+\Sigma)
$$
and
$$
K''\overline S(m'+1)=(\overline y_{\lambda+1}(m'+1)^{\delta_1}\cdots \overline y_{\lambda+s_r}(m'+1)^{\delta_{s_r}}+\Psi_1,\ldots,\overline y_{\lambda+1}(m'+1)^{\delta_1}\cdots \overline y_{\lambda+s_r}(m'+1)^{\delta_{s_r}}+\Psi_e)
$$
where $\Sigma,\Psi_1,\ldots\Psi_e\in (\overline y_1(m'+1),\ldots,\overline y_{\lambda}(m'+1))$ and $\gamma_1,\ldots,\delta_{s_r}\in \NN$. 

Let $\epsilon_i=\max\{\gamma_i,\delta_i\}$  for $1\le i\le s_r$. Define a MTS $\overline S(m'+1)\rightarrow \overline S(m'+2)$ by
$$
\overline y_i(m'+1)=\left\{\begin{array}{ll}
\prod_{j=1}^{s_r}\overline y_{\lambda+j}(m'+2)^{\epsilon_j}\overline y_i(m'+2)&\mbox{ for }1\le i\le \lambda\\
\overline y_i(m'+2)&\mbox{ for }\lambda< i\le n.
\end{array}\right.
$$
Further, by our construction, $\overline S(m'+2)$ is a local ring of the blow up of an ideal $B$ of $S''$ such that 
$$
B\overline S(m'+2)=(\prod_{j=0}^{r-1}\prod_{l=1}^{s_{j+1}}(\overline y_{t_1+\cdots+t_j+l}(m'+2))^{\gamma_{jl}})
$$
for some $\epsilon_{jl}\in \NN$. 
Thus there exists an ideal $J^*$ of $S$ and $\beta>0$ such that $\overline S(m'+2)$ is a local ring of the blow up of $J^*$ and $J^*\overline S(m'+2)=J^{\beta}B\overline S(m'+2)$. 
We  have thus achieved the conclusions of  Theorem \ref{TheoremM3} in $R(m')\rightarrow \overline S(m'+2)$.
\vskip .2truein
We now indicate the changes which need to be made in the statements of \cite{LMTE} to obtain the proof of Theorem \ref{TheoremF}.
\vskip .2truein
We first extend Theorem \ref{TheoremM1} to arbitrary extensions of characteristic zero algebraic function fields. We will call this ``Extended Theorem \ref{TheoremM1}''.

 In the statement and proof of Theorem \ref{TheoremM1}, we replace 
references to Theorem 5.1 \cite{Ast} with Theorem 10.1 \cite{LMTE},  Theorems 4.8 and 4.10 \cite{Ast} with Theorems 9.1 and 9.3 \cite{LMTE}. Also replace $n$ with $m$ when referring to regular parameters in birational extensions of $R$ and $s$ with $\overline s$. We must add the following sentence to the first line of the proof: ``First suppose that $\mbox{rank}(\nu)>0$, so that $\mbox{rank}(\nu)=1$ (here $\nu$ is the restriction of the given rank 1 valuation $\nu^*$ of the quotient field of $S$ to the quotient field $K$ of $R$)''. At the end of the proof, add: ``If $\mbox{rank}(\nu)=0$, then $\nu$ is trivial so $R=K$ and the proof is a substantial simplification.''

\vskip .2truein
We now extend Corollary \ref{CorM2}. We will call this ``Extended Corollary \ref{CorM2}''. In the statement and proof of Corollary \ref{CorM2}, replace references to Theorem 5.1 \cite{Ast} with Theorem 10.1 \cite{LMTE} and Section 4.1 \cite{Ast}  with Section 5 \cite{LMTE}. Replace references to Theorem \ref{TheoremM1} with ``Extended Theorem \ref{TheoremM1}''. Replace $n$ with $m$ when referring to regular parameters in birational extensions of $R$. Replace $s$ with $\overline s$.
We must add the following sentence to the first line of the proof: ``First suppose that $\mbox{rank}(\nu)>0$, so that $\mbox{rank}(\nu)=1$ (here $\nu$ is the restriction of the given rank  1 valuation  $\nu^*$ of the quotient field of $S$ to the quotient field $K$ of $R$)''. At the end of the proof, add: ``If $\mbox{rank}(\nu)=0$, then $\nu$ is trivial so $R=K$ and the proof is a substantial simplification.''

\vskip .2truein
We now adopt the notation on valuation rings introduced on page 1579 - 1581 of \cite{LMTE}, which we will use below.

We  extend Theorem \ref{TheoremM3} to arbitrary extensions of algebraic function fields in the following Theorem \ref{TheoremM4}.

\begin{Theorem}\label{TheoremM4} 
Suppose that $R\rightarrow S$ satisfies the assumptions  of Theorem 10.5 \cite{LMTE} and $I$ is a nonzero ideal in $S$. Then there exist sequences of monoidal transforms $R\rightarrow R'$ and $S\rightarrow S'$ such that $R'\rightarrow S'$ satisfies the conclusions of Theorem 10.5 \cite{LMTE} and
\begin{enumerate}
\item[1)] $S'$ is a local ring of the blowup of an ideal $J$ of $S$ such that 
$$
JS'=(\prod (w_{\overline t_0+\cdots+\overline t_{i-1}+\overline t_{i,1}+\cdots+\overline t_{i,j-1}+l})^{\epsilon_{ijl}})
$$
where the product is over
$$
0\le i\le \beta, 1\le j\le \sigma(i), 1\le l\le \overline s_{ij}
$$
(with $\overline s_{i1}=\overline s_i$) and $\epsilon_{ijl}\in \NN$ for all $i,j,l$.
\item[2)]   We have that
$$
I S'=(\prod (w_{\overline t_0+\cdots+\overline t_{i-1}+\overline t_{i,1}+\cdots+\overline t_{i,j-1}+l})^{\gamma_{ijl}})
$$
where the product is over
$$
0\le i\le \beta, 1\le j\le \sigma(i), 1\le l\le \overline s_{ij}
$$
and $\gamma_{ijl}\in \NN$ for all $i,j,l$.
\end{enumerate}
\end{Theorem}

\begin{proof}
In the  proof of Theorem \ref{TheoremM3}, replace references to Theorem 5.3 \cite{Ast} with Theorem 10.5 \cite{LMTE}. 
Replace references to Theorem \ref{TheoremM3} with  Theorem \ref{TheoremM4}, references to Corollary \ref{CorM2} with Extended Corollary \ref{CorM2}.
and references to Theorem \ref{TheoremM1}  with Extended Theorem \ref{TheoremM1}. The indexing of the variables must be changed (as in the statement of Theorem \ref{TheoremM4}).
The prime ideal $q_{r-1}$ in the valuation ring  of the quotient field of $S$ is replaced by the prime ideal $q_{\beta-1,\sigma(\beta-1)}$ of the valuation ring of the quotient field of $S$, so $q_{r-1}''=q_{r-1}\cap S''$ becomes $q_{\beta-1}''=q_{\beta-1,\sigma(\beta-1)}\cap S''$.

We must add the following sentences as the first lines of the proof:  ``We prove the theorem by induction on $\mbox{rank }V^*$. If $\mbox{rank }V^*=1$ then the theorem is immediate from Extended Corollary \ref{CorM2}. 

By induction on $\gamma=\mbox{rank } V^*$, we may assume that the theorem is true whenever $\mbox{rank } V^*=\gamma-1$. We are reduced to proving the theorem in the following two cases: 
\vskip .2truein
Case 1. $\sigma(\beta)=1$

Case 2. $\sigma(\beta)>1$.
\vskip .2truein

Suppose that we are in Case 1, $\sigma(\beta)=1$. Then $V^*/q_{\beta-1,\sigma(\beta-1)}$ is a rank 1 valuation ring which dominates the rank 1 valuation ring $V/p_{\beta-1}$. $V^*/q_{\beta-1,\sigma(\beta-1)}$ has rational rank $\overline s_{\beta}$ and $V/p_{\beta-1}$ has rational rank $\overline r_{\beta}$.''
\vskip .2truein
We must add the following sentences as the last lines of the proof: ``Now suppose that we are in Case 2, $\sigma(\beta)>1$. Then $V^*/q_{\beta,\sigma(\beta)-1}$ is a rank 1, rational rank $\overline s_{\beta,\sigma(\beta)}$ valuation ring which dominates  the rank 0 valuation ring $V/p_{\beta}$, which is a field. The proof in Case 2 is thus a substantial simplification of the proof in Case 1.''
\end{proof}

\section{A geometric local monomialization theorem}

The main result in this section is Theorem \ref{TheoremM9}.  Uniformizing parameters  on an affine $k$-variety $V$ are a set of elements $u_1,\ldots,u_s\in A=\Gamma(V,\mathcal O_V)$ such that $du_1,\ldots,du_s$ is a free basis of $\Gamma(V,\Omega_{V/k})$ as an $A$-module.

\begin{Theorem}\label{TheoremM6} Suppose that $k$ is a field of characteristic zero, $\phi:Y\rightarrow X$ is a dominant morphism of $k$-varieties  and $\nu$ is a zero dimensional valuation of the function field  $k(Y)$ (the residue field of the valuation ring of $\nu$ is algebraic over $k$) which has a center on $Y$ (the valuation ring of $\nu$ dominates a local ring of $Y$). Further suppose that $\mathcal I\subset \mathcal O_Y$ is a nonzero ideal sheaf. Then there exists a commutative diagram of morphisms of $k$-varieties 
\begin{equation}\label{eqM8}
\begin{array}{lll}
Y_{\nu}&\stackrel{\phi_{\nu}}{\rightarrow}&X_{\nu}\\
\beta_{\nu}\downarrow&&\downarrow \alpha_{\nu}\\
Y&\stackrel{\phi}{\rightarrow}&X
\end{array}
\end{equation}
where the vertical arrows are proper morphisms which are products of blow ups of nonsingular sub varieties, and if $q'$ is the center of $\nu$ on $Y_{\nu}$ and $p'$ is the center of $\nu$ on $X_{\nu}$, then there exists an affine open neighborhood $V_{\nu}$ of $q'$ in $Y_{\nu}$ and an affine open neighborhood $U_{\nu}$ of $p'$ in $X_{\nu,p'}$, regular parameters $y_1,\ldots,y_n$ in $\mathcal O_{Y_{\nu,q'}}$ which are uniformizing parameters on $V_{\nu}$ and regular parameters $x_1,\ldots,x_m$ in $\mathcal O_{X_{\nu},p'}$ which are uniformizing parameters on $U_{\nu}$ (where $m=\dim X$, $n=\dim Y$) and units $\delta_1,\ldots,\delta_m\in \Gamma(V_{\nu},\mathcal O_{Y_{\nu}})$ such that
\begin{enumerate}
\item[1)] $x_i=\prod_{j=1}^ny_j^{c_{ij}}\delta_i\mbox{ with }c_{ij}\in \NN\mbox{ for }1\le i\le m$ and ${\rm rank}(c_{ij})=m$.
\item[2)] $V_{\nu}\setminus Z(y_1\cdots y_n)\rightarrow Y$ is an open immersion;
\item[3)] $\mathcal I\mathcal O_{V_{\nu}}=y_1^{a_1}\cdots y_n^{a_n}\mathcal O_{V_{\nu}}$ for some $a_1,\ldots,a_n\in \NN$
\item[4)] $\phi_{\nu}:V_{\nu}\rightarrow U_{\nu}$ is toroidal (Section 2.2 \cite{ADK}) with respect to the locus of the product of the $y_j$ and the locus of the product of the $x_i$. 
\end{enumerate}
\end{Theorem}

\begin{proof} Let $p$ be the center of $\nu$ on $X$ and $q$ be the center of $\nu$ on $Y$. Let $\nu^*=\nu$, $R=\mathcal O_{X,p}$, $S=\mathcal O_{Y,q}$, $I=\mathcal I_q$,
$K^*=k(Y)$ and $K=k(X)$. Let
$$
\begin{array}{lll}
R'&\rightarrow &S'\\
\uparrow&&\uparrow\\
R&\rightarrow&S
\end{array}
$$
be the diagram of the conclusions of Theorem \ref{TheoremF}. There exists a commutative diagram
$$
\begin{array}{lll}
Y_{\nu}&\stackrel{\phi_{\nu}}{\rightarrow}&X_{\nu}\\
\beta_{\nu}\downarrow&&\downarrow\alpha_{\nu}\\
Y&\stackrel{\phi}{\rightarrow}&X
\end{array}
$$
where the vertical arrows are products of blow ups of nonsingular sub varieties and if $p'$ is the center of $\nu$ on $X_{\nu}$ and $q'$ is the center of $\nu$ on $Y_{\nu}$ then $\mathcal O_{Y_{\nu},q'}=S'$ and $\mathcal O_{X_{\nu},p'}=R'$. Since $k$ has characteristic zero, there exists an affine open  neighborhood $U_{\nu}$ of $p'$ such that $x_1,\ldots,x_m$ are uniformizing parameters on $U_{\nu}$. Let $V_{\nu}$ be an affine neighborhood of $q'$ in $Y_{\nu}$ such that $y_1,\ldots,y_n$ are uniformizing parameters on $V_{\nu}$ and $\delta_1,\ldots,\delta_m\in \Gamma(V_{\nu},\mathcal O_{Y_{\nu}})$ are units. Let $\overline C$ be an $m\times m$ sub matrix of $C=(c_{ij})$ of rank $m$. Without loss of generality,
$$
\overline C=\left(\begin{array}{lll}
c_{11}&\cdots&c_{1m}\\
\vdots&&\vdots\\
c_{m1}&\cdots&c_{mm}
\end{array}\right).
$$
Let $d=\mbox{Det}(\overline C)$.
Let
$$
A=\Gamma(V_{\nu},\mathcal O_{Y_{\nu}})
$$
and
$$
B=A[z_1,\ldots,z_m]/(z_1^d-\delta_1,\ldots,z_m^d-\delta_m)=A[\overline z_1,\ldots,\overline z_m]
$$
where $\overline z_i$ is the class of $z_i$.

Let
$V_{\nu}'=\mbox{Spec}(B)$ with natural finite \'etale morphism $\pi:V_{\nu'}\rightarrow V_{\nu}$. Define $\overline y_j$ for $1\le j\le n$ by
 $$
 y_j=\overline y_j\prod_{l=1}^m\overline z_{l}^{db_{jl}}
 $$
 where
 $$
 \overline B=(b_{jl})=\left(\begin{array}{c}
 -\overline C^{-1}\\
 0\end{array}\right)
 $$
is an $n\times m$ matrix with coefficients in $\frac{1}{d}\ZZ$. We have that for $1\le i\le m$,
\begin{equation}\label{eqU1}
x_i=\prod_{j=1}^ny_j^{c_{ij}}\delta_i=\left(\prod_{j=1}^n\overline y_j^{c_{ij}}\right)\left(\prod_{l=1}^m\overline z_l^{\sum_{j=1}^nc_{ij}db_{jl}}\right)\delta_i
=\prod_{j=1}^n\overline y_j^{c_{ij}}.
\end{equation}
The morphism $\pi:V_{\nu}'\rightarrow V_{\nu}$ is \'etale so $y_1,\ldots,y_n$ are uniformizing parameters on $V_{\nu}'$ or equivalently, $dy_1,\ldots,dy_n$ are a free basis of $\Omega^1_{B/k}$ as a $B$-module. Let 
$$
\epsilon_j=\left(\prod_{l=1}^m\overline z_l^{db_{jl}}\right)^{-1}\in B
$$
and $\overline y_j=\epsilon_jy_j$ for $1\le j\le n$. Since $y_1,\ldots,y_n$ are regular parameters in $\mathcal O_{V'_{\nu},q''}$ for all $q''\in \pi^{-1}(q')$ and $d\overline y_j=\epsilon_jdy_j+y_jd\epsilon_j$, we have that
$$
\left(\Omega^1_{V'_{\nu}/k}/d\overline y_1\mathcal O_{V'_{\nu}}+\cdots+d\overline y_n\mathcal O_{V'_{\nu}}\right)_{q''}=0
$$
for all $q''\in \pi^{-1}(q')$. Let
$$
Z=\mbox{Supp}\left(\Omega^1_{V'_{\nu}/k}/d\overline y_1\mathcal O_{V'_{\nu}}+\cdots+d\overline y_n\mathcal O_{V'_{\nu}}\right).
$$
$Z$ is a closed subset of $V'_{\nu}$ which is disjoint from $\pi^{-1}(q')$. The morphism $V'_{\nu}\rightarrow V_{\nu}$ is finite, so $\pi(Z)$ is a closed subset of $V_{\nu}$ which does not contain $q'$. After replacing $V_{\nu}$ with an affine neighborhood of $q'$ in $V_{\nu}\setminus \pi(Z)$, we have that $\overline y_1,\ldots,\overline y_n$ are uniformizing parameters on $V'_{\nu}$ giving the monomial expression (\ref{eqU1}) and so $V_{\nu}\rightarrow U_{\nu}$ is toroidal  with respect to the locus of the products of the $y_j$ and the products of the $x_i$.
\end{proof}

\begin{Theorem}\label{TheoremM9} Suppose that $k$ is a field of characteristic zero, $\phi:Y\rightarrow X$ is a dominant morphism of $k$-varieties and $\mathcal I\subset \mathcal O_Y$ is a nonzero ideal sheaf. Then there exists a finite number  of commutative diagrams
$$
\begin{array}{lll}
Y_i&\stackrel{\phi_i}{\rightarrow}&X_i\\
\beta_i\downarrow&&\downarrow\alpha_i\\
Y&\stackrel{\phi}{\rightarrow}&X
\end{array}
$$
for $1\le i\le t$ such that the vertical arrows are products of blow ups of nonsingular sub varieties and there are    affine open subsets $V_i\subset Y_i$ and $U_i\subset X_i$ such that $\phi_i(V_i)\subset U_i$, $\cup_{i=1}^t\beta_i(V_i)=Y$ and the restriction $\phi_i:V_i\rightarrow U_i$ is toroidal with respect to strict normal crossings divisor $E_i$ on $V_i$ and $D_i$ on $X_i$ such that  the restriction of $\phi_i$ to $V_i\setminus E_i$ is an open immersion and  $\mathcal I\mathcal O_{V_i}$ is a divisor on $V_i$ whose support is contained in $E_i$.
\end{Theorem}

\begin{proof} Let $\Omega$ be the Zariski Riemann manifold of $Y$ (Section 17 of Chapter VI \cite{ZS2}). The points of $\Omega$ are equivalence classes of valuations of $k(Y)$ which dominate a local ring of $Y$. There are natural continuous surjections $\rho_Z:\Omega\rightarrow Z$ for any birational proper morphism $Z\rightarrow Y$. Let $\Sigma$ be the subset of $\Omega$ of zero dimensional valuations which dominate a local ring of $Y$. For each $\nu\in \Sigma$, we construct a diagram (\ref{eqM8}) with corresponding  open subset $V_{\nu}$ of $Y_{\nu}$. 

Suppose that $\omega$ is a valuation of $k(Y)$ which dominates a local ring of $Y$. If $\omega$ is not zero dimensional, then there exists $\nu\in \Sigma$ such that $\nu$ is composite with $\omega$ (Theorem 7, page 16 \cite{ZS2}), so that $\omega\in \rho_{Y_{\nu}}^{-1}(V_{\nu})$. Thus $\{\rho_{Y_{\nu}}^{-1}(V_{\nu})\mid \nu\in \Sigma\}$ is an open cover of $\Omega$, and  thus there is a finite sub cover since $\Omega$ is quasi compact (Theorem 40 page 113 \cite{ZS2}).

   \end{proof}

\end{document}